\documentclass[11pt,dvips,twoside,letterpaper]{article}
\usepackage[usenames]{color}

\usepackage{pslatex}
\usepackage{fancyhdr}
\usepackage{graphicx}
\usepackage{geometry}

\def\figurename{Figure} % Replace the colon that normally appears after the Figure number by a period.
\makeatletter
\renewcommand{\fnum@figure}[1]{\figurename~\thefigure.}
\makeatother

\def\tablename{Table} % Replace the colon that normally appears after the Figure number by a period.
\makeatletter
\renewcommand{\fnum@table}[1]{\tablename~\thetable.}
\makeatother

\usepackage{amsmath}
\usepackage{amssymb}
\usepackage{amsfonts}
\usepackage{amsthm,amscd}

\newtheorem{theorem}{Theorem}[section]
\newtheorem{lemma}[theorem]{Lemma}
\newtheorem{corollary}[theorem]{Corollary}
\newtheorem{proposition}[theorem]{Proposition}
\theoremstyle{definition}
\newtheorem{definition}[theorem]{Definition}

\newtheorem*{notation}{Notation}

\theoremstyle{remark}
\newtheorem{remark}[theorem]{Remark}

\numberwithin{equation}{section}

\begin{document}
\vskip 0.4in
\title{\bfseries\scshape{Two-scale Convergence of Periodic Elliptic Spectral Problems with Indefinite Density Function in Perforated
Domains}}

\author{\bfseries\scshape Hermann Douanla\thanks{E-mail address: \tt{douanla@chalmers.se}}\\
Department of Mathematical Sciences \\Chalmers University of
Technology\\ Gothenburg, SE-41296, Sweden\\\texttt{ }}
\date{}
\maketitle \thispagestyle{empty}
\begin{abstract} \noindent
Spectral asymptotics of linear periodic
elliptic operators with indefinite (sign-changing) density function is investigated in perforated domains with the two-scale convergence method. The
limiting behavior of positive and negative eigencouples depends crucially on
whether the average of the weight over the solid part is positive,
negative or equal to zero. We prove concise homogenization results in all three cases.
\end{abstract}


\vspace{.08in}


\section{Introduction}\label{s1}
Many nonlinear problems lead, after linearization, to elliptic
eigenvalue problems with an indefinite density function (see e.g.,
the survey paper by de Figueiredo\cite{Figueiredo} and the work of
Hess and Kato\cite{Hess, HessKato}).
 A vast literature in engineering, physics and applied mathematics deals with such problems arising, for instance, in the study
  of transport theory, reaction-diffusion equations and fluid dynamics. In 1904, Holmgren\cite{Holmgren} considered the Dirichlet
   problem $\Delta u+\lambda\rho(x,y)u=0$, on a fixed bounded open set $\Omega\subset\mathbb{R}^2$ when $\rho$ is continuous and changes
    sign; he proved the existence of a double sequence of real eigenvalues of finite multiplicity (one nonnegative and converging
     to $+\infty$, the other one negative and tending to $-\infty$) which can be characterized by the minimax principle. This result
      has been extended to higher dimensions, noncontinuous weight and coefficients in many papers including for example \cite{Bs1, Bs2, Nazarov2}.
       Asymptotic analysis of the eigenvalues has been visited by many mathematicians and is still a hot topic in mathematical analysis.
        Generally speaking, spectral asymptotics is a two folded research area. On the one hand it deals with asymptotic formulas
        (estimates) and asymptotic distribution of the eigenvalues. On the other hand it is concerned with homogenization of eigenvalues
         of oscillating operators on possibly varying domains such as perforated ones. This paper falls within the second framework,
          homogenization theory.

Let $\Omega$ be a bounded domain in $\mathbb{R}^N_x$(the numerical
space of variables $x=(x_1,..., x_N)$, with integer $N\geq 2$) and
let $T\subset Y=(0,1)^N$ be a compact subset of $Y$ in
$\mathbb{R}^N_y$. Unless otherwise specified  we assume that
$\Omega$ and $T$ have $\mathcal{C}^1$ boundaries $\partial\Omega$
and $\partial T$, respectively. For $\varepsilon>0$, we define the
perforated domain $\Omega^\varepsilon$ as follows. we put

\[
t^\varepsilon =\{k\in \mathbb{Z}^N: \varepsilon(k+T)\subset\Omega\},
\]
\[
T^\varepsilon=\bigcup_{k\in t^\varepsilon}\varepsilon(k+T)
\]
and
\[
\Omega^\varepsilon=\Omega\setminus T^\varepsilon.
\]
In this setup, $T$ is the reference hole whereas $\varepsilon(k+T)$
is a hole of size $\varepsilon$ and $T^\varepsilon$ is the
collection of the holes of the perforated domain
$\Omega^\varepsilon$. The family $T^\varepsilon$  is made up with a
finite number of holes since $\Omega$ is bounded. In the sequel, $Y^*$ stands
for $Y\setminus T$ and
$n=(n_i)$ denotes the outer unit normal vector to $\partial
T$ with respect to $Y^*$.
% ------------ [Running Heads - for odd and even pages] - please insert it only on page 2!
\pagestyle{fancy} \fancyhead{} \fancyhead[EC]{Hermann Douanla}
\fancyhead[EL,OR]{\thepage} \fancyhead[OC]{Spectral Asymptotics in Porous Media } \fancyfoot{}
\renewcommand\headrulewidth{0.5pt}
%------------------------------------------------------------------------------

We are interested in the spectral asymptotics (as $\varepsilon\to
0$) of the linear elliptic eigenvalue problem
\begin{equation} \label{eq1.1}
\left\{\begin{aligned} -\sum_{i,j=1}^N\frac{\partial}{\partial
x_j}\left(a_{ij}(\frac{x}{\varepsilon})\frac{\partial
u_\varepsilon}{\partial x_i}\right)&=\rho(\frac{x}{\varepsilon})\lambda_\varepsilon u_\varepsilon\text{ in } \Omega^\varepsilon\\
\sum_{i,j=1}^N a_{ij}(\frac{x}{\varepsilon})\frac{\partial
u_\varepsilon}{\partial x_j}n_i(\frac{x}{\varepsilon})&=0 \text{ on }\partial
T^\varepsilon\\
u_\varepsilon&=0 \text{ on } \partial \Omega,
\end{aligned}\right.
\end{equation}
where $a_{ij}\in L^\infty(\mathbb{R}^N_y)$ ($ 1\leq i,j\leq N$), with the symmetry condition
$a_{ji}=a_{ij}$, the $Y$-periodicity hypothesis: for every
$k\in\mathbb{Z}^N$ one has $a_{ij}(y+k)=a_{ij}(y)$ almost everywhere
in $y\in\mathbb{R}^N_y $, and finally the (uniform) ellipticity condition:
there exists $\alpha>0$ such that
\begin{equation}\label{eq1.2}
 \sum_{i,j=1}^{N}a_{ij}(y)\xi_j\xi_i\geq\alpha|\xi|^2
\end{equation} for all $\xi\in\mathbb{R}^N$ and for almost all $y\in\mathbb{R}^N_y$, where $|\xi|^2=|\xi_1|^2+\cdots +|\xi_N|^2$.
The density function $\rho\in L^\infty(\mathbb{R}^N_y) $ is
$Y$-periodic and changes sign on $Y^*$, that is, both the set
$\{y\in Y^*, \rho(y)<0\}$ and $\{y\in Y^*, \rho(y)>0\}$ are of
positive Lebesgue measure. This hypothesis makes the problem under
consideration nonstandard. As stated above, it is
well known (see \cite{Holmgren, Nazarov2}) that under the preceding
hypotheses, for each $\varepsilon>0$ the spectrum of (\ref{eq1.1})
is discrete and consists of two infinite sequences
$$
0<\lambda_\varepsilon^{1,+}\leq \lambda_\varepsilon^{2,+}\leq \cdots
\leq \lambda_\varepsilon^{n,+}\leq \dots,\quad \lim_{n\to
+\infty}\lambda_\varepsilon^{n,+}=+\infty
$$
and
$$
0>\lambda_\varepsilon^{1,-}\geq \lambda_\varepsilon^{2,-}\geq \cdots
\geq \lambda_\varepsilon^{n,-}\geq \dots,\quad \lim_{n\to
+\infty}\lambda_\varepsilon^{n,-}=-\infty.
$$
The asymptotic behavior of the eigencouples depends crucially on whether the average of $\rho$ over $Y^*$,
$M_{Y^*}(\rho)=\int_{Y^*}\rho(y)dy$, is positive, negative or equal to zero. All three cases are carefully investigated in this paper.

The homogenization of spectral problems has been widely explored. In
a fixed domain, homogenization of spectral problems with point-wise
positive density function goes back to Kesavan \cite{Kesavan1,
Kesavan2}. In perforated domains, spectral asymptotics was first
considered by  Rauch and Taylor\cite{Rauch, Taylor} but the first
homogenization result in that direction pertains to
Vanninathan\cite{Vanni}. Since then a lot has been written on
spectral asymptotics in perforated media, we mention the works
\cite{Kaizu, Pastukhova, Roppongi2} and the references therein to
cite a few. Homogenization of elliptic operators with sing-changing
density function in a fixed domain has been investigated by Nazarov
et al. \cite{Nazarov1, Nazarov2, Nazarov3} via a combination of
formal asymptotic expansion and Tartar's energy method. Recently,
the Two-scale convergence method has been utilized to handle the
homogenization process for some eigenvalue problems  with constant
density function\cite{douanla1, douanla2} and sign-changing density
function\cite{douanla3}.

In this paper we investigate in periodically perforated domains the spectral
asymptotics of periodic elliptic linear differential operators of
order two in divergence form with a sing-changing density function.
We obtain accurate and concise homogenization results in all three cases: $ M_{Y^*}(\rho)>0 $ (Theorem
\ref{t3.1} and Theorem \ref{t3.2}),\ $M_{Y^*}(\rho)=0$ (Theorem \ref{t3.3}) and  $M_{Y^*}(\rho)<0$ (Theorem \ref{t3.1}
 and Theorem \ref{t3.2}), by using the two-scale
convergence method\cite{AB, GNWL, G89, Zhikov} introduced by
Nguetseng\cite{G89} and further developed by Allaire\cite{AB}.
Namely, if $M_{Y^*}(\rho)>0$
 then the positive eigencouples behave like in the case of point-wise positive density function, i.e.,
for $k\geq 1$, $\lambda^{k,+}_\varepsilon$ converges as $\varepsilon\to 0$ to the $k^{th}$ eigenvalue of the limit
 spectral problem on $\Omega$, corresponding extended eigenfunctions converge along subsequences. As regards the
  "negative" eigencouples, $\lambda^{k,-}_\varepsilon$ converges to $-\infty$ at the rate $\frac{1}{\varepsilon^2}$ and the corresponding eigenfunctions
   oscillate rapidly. We use a factorization technique (\cite{Nazarov3, Vanni}) to prove convergence
    of $\{\lambda^{k,-}_\varepsilon-\frac{1}{\varepsilon^2}\lambda_1^-\}$ - where ($\lambda_1^-, \theta_{1}^-)$
    is the first negative eigencouple to a local spectral problem - to the $k^{th}$ eigenvalue of  a limit spectral problem
     which is different from that obtained for positive eigenvalues. As regards eigenfunctions,
    extensions of  $\{\frac{u^{k,-}_\varepsilon}{(\theta_{1}^-)^\varepsilon}\}_{\varepsilon\in E}$ - where $(\theta_{1}^-)^\varepsilon(x)=\theta_{1}^-(\frac{x}{\varepsilon})$ -
      converge along subsequences to the $k^{th}$ eigenfunctions of the limit problem.
In the case when $M_{Y^*}(\rho)=0$, $\lambda^{k,\pm}_\varepsilon$ converges to $\pm\infty$ at the rate $\frac{1}{\varepsilon}$ and the limit spectral problem generates a quadratic operator pencil. We prove that $\varepsilon\lambda^{k,\pm}_\varepsilon$ converges to the $(k,\pm)^{th}$ eigenvalue of the limit operator, extended eigenfunctions converge along
 subsequences as well. The case when $M_{Y^*}(\rho)<0$ is equivalent to that when $M_{Y^*}(\rho)>0$, just replace $\rho$ with $-\rho$.
 The reader may consider the reiteration procedure in multiscale periodically perforated domains to have some fun.

Unless otherwise specified, vector spaces throughout are considered
over $\mathbb{R}$, and scalar functions are assumed to take real
values. We will make use of the following notations. Let
$F(\mathbb{R}^N)$ be a given function space. We denote by
$F_{per}(Y)$ the space of functions in $F_{loc}(\mathbb{R}^N)$ (when
it makes sense) that are $Y$-periodic, and by
$F_{per}(Y)/\mathbb{R}$ the space of those functions $u\in
F_{per}(Y)$ with $\int_Y u(y)dy=0$. We denote by $H^1_{per}(Y^*)$
the space of functions in $H^1(Y^*)$ assuming same values on the
opposite faces of $Y$ and $H^1_{per}(Y^*)/\mathbb{R}$ stands for the
subset of $H^1_{per}(Y^*)$ made up of functions $u\in
H^1_{per}(Y^*)$ verifying $\int_{Y^*}u(y)dy=0$. Finally, the letter
$E$ denotes throughout a family of  strictly positive real numbers
$(0<\varepsilon<1)$ admitting $0$ as accumulation point. The
numerical space $\mathbb{R}^N$ and its open sets are provided with
the Lebesgue measure denoted by $dx=dx_1...dx_N$. The usual gradient
operator will be denoted by $D$. The rest of the paper is organized
as follows. Section \ref{s2} deals with some preliminary results
while homogenization processes are considered in Section \ref{s3}.

\section{Preliminaries}\label{s2}
We first recall the definition and the main compactness theorems of
the two-scale convergence method. Throughout this section, $\Omega$
is a smooth open bounded set in $\Bbb{R}^N_x$ (integer $N\geq 2$)
and $Y=(0,1)^N$ is the unit cube.

\begin{definition}\label{d2.1}
A sequence $(u_\varepsilon)_{\varepsilon\in E}\subset L^2(\Omega)$
is said to two-scale converge in $L^2(\Omega)$ to some $u_0\in
L^2(\Omega\times Y)$ if as $E\ni\varepsilon\to 0$,
\begin{equation}\label{eq2.14}
    \int_\Omega u_\varepsilon(x)\phi(x,\frac{x}{\varepsilon})dx\to \iint_{\Omega\times Y}u_0(x,y)\phi(x,y)dxdy
\end{equation}
for all $\phi\in L^2(\Omega;\mathcal{C}_{per}(Y))$.
\end{definition}
\begin{notation}
We express this by writing $u_\varepsilon\xrightarrow{2s}u_0$ in
$L^2(\Omega)$.
\end{notation}
The following compactness theorems (see \cite{AB, G89, GW2007}) are cornerstones of the two-scale convergence
method.
\begin{theorem}\label{t2.1}
Let $(u_\varepsilon)_{\varepsilon\in E}$ be a bounded sequence in
$L^2(\Omega)$. Then a subsequence $E'$ can be extracted from  $E$
such that as $E'\ni\varepsilon\to 0$, the sequence
$(u_\varepsilon)_{\varepsilon\in E'}$ two-scale converges in
$L^2(\Omega)$ to some $u_0\in L^2(\Omega\times Y)$.
\end{theorem}

\begin{theorem}\label{t2.21}
Let $(u_\varepsilon)_{\varepsilon\in E}$ be a bounded sequence in
$H^1(\Omega)$. Then a subsequence $E'$ can be extracted from  $E$
such that as $E'\ni\varepsilon\to 0$
\begin{eqnarray}
  u_\varepsilon &\to& u_0 \ \ \  \text{ in } H^1(\Omega)\text{-weak}\label{eq1.1111}\\
  u_\varepsilon &\to&u_0  \ \ \ \ \ \ \ \ \text{ in } L^2(\Omega)\label{eq1.1222}\\
  \frac{\partial u_\varepsilon}{\partial x_j}&\xrightarrow{2s}&\frac{\partial u_0}{\partial x_j}+\frac{\partial u_1}{\partial y_j}
  \ \ \ \ \text{ in } L^2(\Omega) \ \ (1\leq j \leq N)\label{eq1.1333}
\end{eqnarray}
where $u_0\in H^1(\Omega)$ and $u_1\in L^2(\Omega;H^1_{per}(Y))$.
Moreover, as $E'\ni\varepsilon\to 0$ we have
\begin{equation}\label{eq1.1444}
\int_{\Omega}\frac{u_\varepsilon (x)}{\varepsilon}\psi(x,\frac{x}{\varepsilon})dx\to
\iint_{\Omega\times Y}u_1(x,y)\psi(x,y)dx\,dy
\end{equation}
for $\psi\in \mathcal{D}(\Omega)\otimes (L^2_{per}(Y)/\mathbb{R})$.
\end{theorem}

\begin{proof}
The first part (\ref{eq1.1111})-(\ref{eq1.1333}) is classical (see \cite{AB, G89}). The second part, (\ref{eq1.1444}), was proved in
 \cite{GW2007} in the general framework  of deterministic homogenization but as it is of great importance in this paper and for the sake
  of completeness, we provide its proof in the periodic setting. Let $\psi=(\varphi, \theta)\in \mathcal{D}(\Omega)\times (L^2_{per}(Y)/\mathbb{R})$.
   By the mean value
   zero condition over $Y$ for $\theta$ we conclude that there exists
    a unique solution $\vartheta\in H^1_{per}(Y)/\mathbb{R}$ to
\begin{equation*}
\left\{\begin{aligned} &\Delta_y \vartheta=\theta\quad \text{in} \ Y\\
&\vartheta\in  H^1_{per}(Y)/\mathbb{R}.
\end{aligned}\right.
\end{equation*}
Put $\phi=D_y \vartheta$. We get
\begin{eqnarray*}
  \int_\Omega \frac{u_\varepsilon(x)}{\varepsilon}\psi(x,\frac{x}{\varepsilon})dx&=& \int_\Omega \frac{u_\varepsilon(x)}{\varepsilon}
  \varphi(x)\theta(\frac{x}{\varepsilon})dx= \\
   \int_\Omega u_\varepsilon(x)
  \varphi(x)\text{div}_x\phi(\frac{x}{\varepsilon})dx&=&-\int_\Omega D_x(u_\varepsilon(x)\varphi(x))\cdot\phi(\frac{x}{\varepsilon})dx
\end{eqnarray*}
A limit passage $(\varepsilon\to 0)$ using (\ref{eq1.1333}) yields
\begin{eqnarray*}
 \lim_{\varepsilon\to 0}\int_{\Omega}\frac{u_\varepsilon (x)}{\varepsilon}\psi(x,\frac{x}{\varepsilon})dx&=&
 -\iint_{\Omega\times Y}[D_x u_0(x)+D_y u_1(x,y)]\varphi(x)\cdot\phi(y)dydx\\ &=& -\iint_{\Omega\times Y}D_y u_1(x,y)\varphi(x)\cdot\phi(y)dydx \\
  &=& \iint_{\Omega\times Y} u_1(x,y)\varphi(x)\text{div}_y \phi(y)dydx\\&=& \iint_{\Omega\times Y} u_1(x,y)\psi(x,y)dydx.
\end{eqnarray*}
This completes the proof.
\end{proof}

\begin{remark}
In Theorem \ref{t2.21} the function $u_1$ is unique up to an additive function of variable $x$. We need to fix its choice according to our
future needs. To do this, we introduce the following space
\[
H^{1,*}_{per}(Y)=\{u\in H^1_{per}(Y):\int_{Y^*} u(y)dy=0\}.
\]
\noindent This defines a closed subspace of $H^1_{per}(Y)$ as it is the kernel of the bounded linear functional
$u\mapsto\int_{Y^*}u(y)dy$ defined on $H^1_{per}(Y)$. It is to be noted that for $u\in H^{1,*}_{per}(Y)$,
its restriction to $Y^*$ (which will still be denoted by $u$ in the sequel) belongs to $H^1_{per}(Y^*)/\mathbb{R}$.
\end{remark}
 We will use the following version of Theorem \ref{t2.21}.

\begin{theorem}\label{t2.2}
Let $(u_\varepsilon)_{\varepsilon\in E}$ be a bounded sequence in
$H^1(\Omega)$. Then a subsequence $E'$ can be extracted from  $E$
such that as $E'\ni\varepsilon\to 0$
\begin{eqnarray}
  u_\varepsilon &\to& u_0 \ \ \  \text{ in } H^1(\Omega)\text{-weak}\label{eq1.11}\\
  u_\varepsilon &\to& u_0  \ \ \ \ \ \ \ \ \text{ in } L^2(\Omega)\label{eq1.12}\\
  \frac{\partial u_\varepsilon}{\partial x_j}&\xrightarrow{2s}&\frac{\partial u_0}{\partial x_j}+\frac{\partial u_1}{\partial y_j}
  \ \ \ \ \text{ in } L^2(\Omega) \ \ (1\leq j \leq N)\label{eq1.13}
\end{eqnarray}
where $u_0\in H^1(\Omega)$ and $u_1\in L^2(\Omega;H^{1,*}_{per}(Y))$.
Moreover, as $E'\ni\varepsilon\to 0$ we have
\begin{equation}\label{eq1.14}
\int_{\Omega}\frac{u_\varepsilon (x)}{\varepsilon}\psi(x,\frac{x}{\varepsilon})dx\to
\iint_{\Omega\times Y}u_1(x,y)\psi(x,y)dx\,dy
\end{equation}
for $\psi\in \mathcal{D}(\Omega)\otimes (L^2_{per}(Y)/\mathbb{R})$.
\end{theorem}
\begin{proof}
Let $\widetilde{u}_1\in L^2(\Omega; H^1_{per}(Y))$ be such that Theorem \ref{t2.21} holds with $\widetilde{u}_1$ in place of $u_1$. Put
\[
u_1(x,y)=\widetilde{u}_1(x,y)-\frac{1}{|Y^*|}\int_{Y^*}\widetilde{u}_1(x,y)dy\qquad (x,y)\in\Omega\times Y,
\]
where $|Y^*|$ stands for the Lebesgue measure of $Y^*$. Then $u_1\in L^2(\Omega; H^{1,*}_{per}(Y))$
and moreover $D_y u_1= D_y\widetilde{u}_1$ so that (\ref{eq1.13}) holds.
\end{proof}

We now gather some preliminary results we will need in our
homogenization processes. We introduce the characteristic function $\chi_{G}$ of
$$
G=\mathbb{R}^N_y\setminus \Theta
$$
with
$$
\Theta=\bigcup_{k\in \mathbb{Z}^N}(k+T).
$$
It follows from the closeness of $T$ that $\Theta$ is closed in
$\mathbb{R}^N_y$ so that $G$ is an open subset of $\mathbb{R}^N_y$.
Next, let $\varepsilon\in E$ be arbitrarily fixed and define
\[
V_\varepsilon =\{u\in H^1(\Omega^\varepsilon) : u=0 \text{ on }
\partial\Omega\}.
\]
We equip $V_\varepsilon$ with the $H^1(\Omega^\varepsilon)$-norm
which makes it a Hilbert space. We recall the following classical
extension result \cite{CSJP}.
\begin{proposition}\label{p3.1}
For each $\varepsilon\in E$ there exists an operator $P_\varepsilon$
of  $V_\varepsilon$ into $H^1_0(\Omega)$ with the following
properties:
\begin{itemize}
    \item $P_\varepsilon$ sends continuously and linearly $V_\varepsilon$ into $H^1_0(\Omega)$.
    \item $(P_\varepsilon v)|_{\Omega^\varepsilon}=v$ for all $v\in V_\varepsilon$.
    \item $\|D(P_\varepsilon v)\|_{L^2(\Omega)^N}\leq c \|D v\|_{L^2(\Omega^\varepsilon)^N}$ for all
    $v\in V_\varepsilon$, where $c$ is a constant independent of $\varepsilon$.
\end{itemize}
\end{proposition}
In the sequel, we will explicitly write  the just-defined extension
operator everywhere needed but we will abuse notations on the local
extension operator (see \cite{CSJP} for its definition): the
extension to $Y$ of $u\in H^1_{per}(Y^*)/\mathbb{R}$ will still be
denoted by $u$ (this extension is an element of $H^{1,*}_{per}(Y)$).

Now, let $Q^\varepsilon=\Omega\setminus (\varepsilon \Theta)$. This
is an open set in $\mathbb{R}^N$ and $\Omega^\varepsilon\setminus
Q^\varepsilon$ is the intersection of $\Omega$ with the collection
of the holes crossing the boundary $\partial\Omega$. We have the
following result which implies that the holes crossing the boundary
$\partial\Omega$ are of no effects as regards the homogenization
processes since they are in arbitrary narrow stripe along the
boundary.

 \begin{lemma}\cite{Gper}\label{l3.1}
 Let $K\subset\Omega$ be a compact set independent of $\varepsilon$. There is some $\varepsilon_0>0$ such
  that $\Omega^\varepsilon\setminus Q^\varepsilon\subset\Omega\setminus K$ for any $0<\varepsilon\leq\varepsilon_0$.
 \end{lemma}
Next, we introduce the space
\[
\mathbb{F}^1_0=H^1_0(\Omega)\times L^2\left(\Omega;
H^{1,*}_{per}(Y)\right).
\]
Endowed with the following norm
\[
\|\textbf{v}\|_{\mathbb{F}^1_0}= \left\|D_x v_0+D_y
v_1\right\|_{L^2(\Omega\times Y)}\ \ \ \
(\textbf{v}=(v_0,v_1)\in\mathbb{F}^1_0),
\]
$\mathbb{F}^1_0$ is a Hilbert space admitting $F_{0}^{\infty
}=\mathcal{D}(\Omega )\times [\mathcal{D}(\Omega )\otimes
\mathcal{C}_{per}^{\infty,*}(Y)]$ (where
$\mathcal{C}_{per}^{\infty,*}(Y)=\{u\in\mathcal{C}_{per}^{\infty}(Y):\int_{Y^*}u(y)dy=0\}$)
as a dense subspace. This being so, for
$(\textbf{u},\textbf{v})\in\mathbb{F}^1_0\times\mathbb{F}^1_0$, let
\begin{equation*}
    a_\Omega(\textbf{u},\textbf{v})=\sum_{i,j=1}^N\iint_{\Omega \times Y^*}a_{ij}(y)\left(\frac{\partial u_0}{\partial x_j}+
    \frac{\partial u_1}{\partial y_j}\right)\left(\frac{\partial v_0}{\partial x_i}+
    \frac{\partial v_1}{\partial y_i}\right)\,dxdy.
 \end{equation*}
This define a symmetric, continuous bilinear form on
$\mathbb{F}^1_0\times \mathbb{F}^1_0$. We will need the following
results whose proof can be found in \cite{douanla1}.
\begin{lemma}\label{l3.2}
Fix $\Phi=(\psi _{0},\psi _{1})\in F_0^\infty$ and define
$\Phi_\varepsilon:\Omega\to \mathbb{R}$ ($\varepsilon>0$) by
\begin{equation*}
    \Phi_\varepsilon(x)=\psi_0(x) +
    \varepsilon\psi_1(x,\frac{x}{\varepsilon})\quad (x\in \Omega).
\end{equation*}
If $(u_\varepsilon)_{\varepsilon\in E}\subset H^1_0(\Omega)$ is such
that
\begin{equation*}
    \frac{\partial u_\varepsilon}{\partial x_i}\xrightarrow{2s} \frac{\partial u_0}{\partial x_i}+\frac{\partial u_1}
    {\partial y_i}\ \ \text{ in }\ \ L^2(\Omega) \ (1\leq i\leq N)
\end{equation*}
as $E\ni \varepsilon\to 0$ for some $\textbf{u}=(u_0, u_1)\in
\mathbb{F}^1_0$, then
 \begin{equation*}
    a^\varepsilon(u_\varepsilon,\Phi_\varepsilon)\to a_\Omega(\textbf{u},\Phi)
 \end{equation*}
as  $E\ni \varepsilon\to 0$, where
\[
 a^\varepsilon(u_\varepsilon,\Phi_\varepsilon)=\sum_{i,j=1}^N\int_{\Omega^\varepsilon} a_{ij}
 (\frac{x}{\varepsilon})\frac{\partial u_\varepsilon}{\partial x_j}
\frac{ \partial \Phi_\varepsilon}{\partial x_i}dx.
\]
\end{lemma}

We now construct and point out the main properties of the so-called
homogenized coefficients. We put
\begin{equation}\label{eq3.1011}
    a(u,v)=\sum_{i,j=1}^N\int_{Y^*}a_{ij}(y)\frac{\partial u}{\partial y_j}\frac{\partial v}{\partial
    y_i}dy,
\end{equation}
\begin{equation*}
  \qquad\qquad  l_j(v)=\sum_{k=1}^N\int_{Y^*}a_{kj}(y)\frac{\partial
    v}{\partial
    y_k}dy\quad (1\leq j\leq N)
\end{equation*}
and
\begin{equation*}
    l_0(v)=\int_{Y^*}\rho(y)v(y)dy
\end{equation*}
 for $u,v\in H^1_{per}(Y^*)/\mathbb{R}$. Equipped with the norm
\begin{equation}\label{eq3.311}
    \|u\|_{H^1_{per}(Y^*)/\mathbb{R}}=\|D_y u\|_{L^2(Y^*)^N}\ \ \ (u\in H^1_{per}(Y^*)/\mathbb{R}),
\end{equation}
$H^1_{per}(Y^*)/\mathbb{R}$ is a Hilbert space.

\begin{proposition}\label{p3.2}
Let $1\leq j\leq N$.  The local variational problems
\begin{equation}\label{eq3.9}
  u\in H^1_{per}(Y^*)/\mathbb{R} \text{ and } a(u,v)=l_j(v)\  \text{ for all }\  v\in H^1_{per}(Y^*)/\mathbb{R}
\end{equation}
and
\begin{equation}\label{eq3.91}
  u\in H^1_{per}(Y^*)/\mathbb{R} \text{ and } a(u,v)=l_0(v)\  \text{ for all }\
v\in H^1_{per}(Y^*)/\mathbb{R}
\end{equation}
admit each a unique solution, assuming for (\ref{eq3.91}) that $M_{Y^*}(\rho)=0$.
\end{proposition}

Let $1\leq i,j\leq N$. The homogenized coefficients read
\begin{equation}\label{eq3.10}
    q_{ij}=\int_{Y^*}a_{ij}(y)dy-\sum_{l=1}^N\int_{Y^*}a_{il}(y)\frac{\partial\chi^j}{\partial
    y_l}(y)dy
\end{equation}
where $\chi^j\ \ (1\leq j\leq N)$ is the solution to (\ref{eq3.9}).
We recall that
 $q_{ji}=q_{ij}\ \ (1\leq i,j\leq N)$ and there exists a constant $\alpha_{0}>0$ such that
 \[
   \sum_{i,j=1}^{N}q_{ij}\xi_{j}\xi_{i}\geq\alpha_{0}|\xi|^{2}
\]
for all $\xi\in \mathbb{R}^{N}$ (see e.g., \cite{BLP}).

We now say a few words on the existence result for (\ref{eq1.1}). The weak formulation of (\ref{eq1.1}) reads:
Find $(\lambda_\varepsilon, u_\varepsilon)\in\mathbb{C}\times V_\varepsilon$, ($u_\varepsilon\neq 0$) such that
\begin{equation}\label{eq3.101}
    a^\varepsilon(u_\varepsilon,v)=\lambda_\varepsilon(\rho^\varepsilon u_\varepsilon, v)_{\Omega^\varepsilon}, \quad v\in V_\varepsilon,
\end{equation}
where
$$
(\rho^\varepsilon u_\varepsilon, v)_{\Omega^\varepsilon}=\int_{\Omega^\varepsilon}\rho^\varepsilon u_\varepsilon v dx.
$$
Since $\rho^\varepsilon$ changes sign, the classical results on the spectrum of semi-bounded self-adjoint operators
with compact resolvent do not apply. To handle this, we follow the ideas in \cite{Nazarov3}. The bilinear form
$(\rho^\varepsilon u, v)_{\Omega^\varepsilon}$ defines a bounded linear operator $K^\varepsilon:V_\varepsilon\to V_\varepsilon$ such that
$$
(\rho^\varepsilon u, v)_{\Omega^\varepsilon}=a^\varepsilon(K^\varepsilon u,v) \quad (u,v\in V_\varepsilon).
$$
The operator $K^\varepsilon$ is symmetric and its domains $D(K^\varepsilon)$ coincides with the whole $V_\varepsilon$,
thus it is self-adjoint.  Recall that the gradient norm is equivalent to the $H^1(\Omega^\varepsilon)$-norm on $V_\varepsilon$.
 Looking at $K^\varepsilon u$ as the solution to the boundary value problem
\begin{equation}\label{}\left\{\begin{aligned}
 -div(a(\frac{x}{\varepsilon})D_x(K^\varepsilon u))&=\rho^\varepsilon u \quad \text{in }\Omega^\varepsilon\\
 a(\frac{x}{\varepsilon})D_xK^\varepsilon u\cdot n(\frac{x}{\varepsilon})&=0 \quad \text{ on  } \partial T^\varepsilon\\
 K^\varepsilon u(x)&=0 \quad \text{ on  } \partial \Omega,
\end{aligned}\right.
\end{equation}
we get a constant $C_\varepsilon>0$ such that $\|K^\varepsilon
u\|_{V^\varepsilon}\leq
C_\varepsilon\|u\|_{L^2(\Omega^\varepsilon)}$. As $V^\varepsilon$ is
compactly embedded in $L^2(\Omega^\varepsilon)$ (indeed,
$H^1(\Omega^\varepsilon)\hookrightarrow L^2(\Omega^\varepsilon)$ is
compact as $\partial \Omega^\varepsilon$ is $\mathcal{C}^1$), the
operator $K^\varepsilon$ is compact. We can rewrite (\ref{eq3.101})
as follows
$$
K^\varepsilon u_\varepsilon=\mu_\varepsilon u_\varepsilon,\quad \mu_\varepsilon=\frac{1}{\lambda_\varepsilon}.
$$
Notice that (see e.g., \cite{Bs3}) in the case $\rho\geq 0$ in $Y$, the operator $K^\varepsilon$ is positive and its
 spectrum $\sigma(K^\varepsilon)$ lives in $[0, \|K^\varepsilon\|]$ and $\mu_\varepsilon=0$ belongs to the essential
 spectrum $\sigma_e(K^\varepsilon)$. The essential spectrum of a self-adjoint operator $L$ is by definition
  $\sigma_e(L)=\sigma_p^\infty(L)\cup\sigma_c(L)$, where $\sigma_p^\infty(L)$ is the set of eigenvalues of infinite
   multiplicity and $\sigma_c(L)$ is the continuous spectrum. The spectrum of $K^\varepsilon$ is described by the
   following proposition whose proof is omitted since similar to that of \cite[Lemma 1]{Nazarov3}.
\begin{lemma}\label{l2.1}
Let $\rho\in L^\infty_{per}(Y)$ be such that the sets $\{y\in Y^*: \rho(y)< 0\}$ and $\{y\in Y^*: \rho(y)> 0\}$
are both of positive Lebesgue measure. Then for any $\varepsilon>0$, we have $\sigma(K^\varepsilon)\subset [-\|K^\varepsilon\|, \|K^\varepsilon\|]$
 and $\mu=0$ is the only element of the essential spectrum $\sigma_e(K^\varepsilon)$. Moreover, the discrete spectrum of $K^\varepsilon$ consists
 of two infinite sequences
\begin{eqnarray*}
  \mu_\varepsilon^{1,+}\geq \mu_\varepsilon^{2,+}\geq \cdots \geq\mu_\varepsilon^{k,+}\geq\cdots\to 0^+,&& \\
 \mu_\varepsilon^{1,-}\leq \mu_\varepsilon^{2,-}\leq \cdots \leq\mu_\varepsilon^{k,-}\leq\cdots\to 0^-.&& \\
 \end{eqnarray*}
\end{lemma}
\begin{corollary}\label{c11}
The hypotheses are those of Lemma \ref{l2.1}. Problem (\ref{eq1.1}) has a discrete set of eigenvalues consisting of two sequences
\begin{eqnarray*}
  0<\lambda_\varepsilon^{1,+}\leq\lambda_\varepsilon^{2,+}\leq \cdots\leq \lambda_\varepsilon^{k,+}\leq\cdots \to +\infty, && \\
   0>\lambda_\varepsilon^{1,+}\geq\lambda_\varepsilon^{2,-}\geq \cdots\geq \lambda_\varepsilon^{k,-}\geq\cdots \to -\infty. &&
\end{eqnarray*}
\end{corollary}

We are now in a position to state the main results of this paper.
\section{Homogenization results}\label{s3}
In this section we state and prove homogenization results for both cases $M_{Y^*}(\rho)>0$ and  $M_{Y^*}(\rho)=0$.
The homogenization results in the case when $M_{Y^*}(\rho)<0$ can be deducted from the case $M_{Y^*}(\rho)>0$ by replacing $\rho$ with $-\rho$.
We start with the less technical case.
\subsection{The case $M_{Y^*}(\rho)>0$}
We start with the homogenization result for the positive part of the
spectrum $(\lambda_\varepsilon^{k,+}, u_\varepsilon^{k,+})_{\varepsilon\in E}$.
\subsubsection{Positive part of the spectrum}
We assume (this is not a restriction) that the corresponding eigenfunctions
are orthonormalized as follows
\begin{equation}\label{eq3.131}
    \int_{\Omega^\varepsilon}\rho(\frac{x}{\varepsilon})u_\varepsilon^{k,+}u_\varepsilon^{l,+}dx=\delta_{k,l}\quad
    k,l=1,2,\cdots
\end{equation}
The homogenization results states as

\begin{theorem}\label{t3.1}
We assume that $\Omega$ and $T$ have $\mathcal{C}^1$ boundaries. For each $k\geq 1$ and each $\varepsilon\in E$, let
$(\lambda^{k,+}_\varepsilon,u^{k,+}_\varepsilon)$ be the $k^{th}$
positive eigencouple to (\ref{eq1.1}) with $M_{Y^*}(\rho)>0$ and
(\ref{eq3.131}). Then, there exists a subsequence $E'$ of $E$ such
that
\begin{eqnarray}
 \lambda^{k,+}_\varepsilon &\to&  \lambda^{k}_0\quad\text{in }\ \mathbb{R}\ \text{ as }E\ni\varepsilon\to 0\label{eq3.14}\\
  P_\varepsilon u^{k,+}_\varepsilon&\to& u^{k}_0 \quad\text{in }\ \ H^1_0(\Omega)\text{-weak}\text{ as }E'\ni\varepsilon\to 0\label{eq3.15}\\
  P_\varepsilon u^{k,+}_\varepsilon&\to& u^{k}_0 \quad\text{in }\ \ L^2(\Omega)\text{ as }E'\ni\varepsilon\to 0\label{eq3.16}\\
\frac{\partial P_\varepsilon u^{k,+}_\varepsilon}{\partial
x_{j}}&\xrightarrow{2s} & \frac{\partial u_{0}^{k}}{\partial
x_{j}}+\frac{\partial u_{1}^{k}}{\partial y_{j}}\text{\ in
}L^{2}(\Omega)\text{ as }E'\ni\varepsilon\to 0\ (1\leq j\leq
N)\label{eq3.17}
\end{eqnarray}
where $(\lambda^{k}_0,u^{k}_0)\in \mathbb{R}\times
H^1_0(\Omega)$ is the $k^{th}$ eigencouple to the
 spectral problem
\begin{equation}\label{eq3.18}\left\{\begin{aligned}
 -\sum_{i,j=1}^N\frac{\partial}{\partial x_i}\left(\frac{1}{M_{Y^*}(\rho)}q_{ij}\frac{\partial u_0}{\partial x_j}\right)&
 =\lambda_0 u_0\quad \text{in }\Omega\\
 u_0&=0\quad \text{ on  } \partial\Omega\\
 \int_{\Omega}|u_0|^2dx&=\frac{1}{M_{Y^*}(\rho)},
\end{aligned}\right.
\end{equation}
$u_1^{k}\in L^2(\Omega;H^{1,*}_{per}(Y))$ and where the coefficients $\{q_{ij}\}_{1\leq i,j\leq N}$ are defined by
(\ref{eq3.10}). Moreover, for almost every $x\in\Omega$ the following hold true:\\
\textbf{(i)} \ The restriction to $Y^*$ of $u_1^{k}(x)$ is the solution to the
variational problem
\begin{equation}\label{eq3.21}\left\{\begin{aligned}
  &u_1^{k}(x)\in  H^1_{per}(Y^*)/\mathbb{R}\\&
  a(u_1^{k}(x),v)=-\sum_{i,j=1}^N \frac{\partial u^{k}_0}{\partial
  x_j}  \int_{Y^*}a_{ij}(y)\frac{\partial v}{\partial y_i}dy \\&
   \forall v\in H^1_{per}(Y^*)/\mathbb{R},
   \end{aligned}\right.
\end{equation}
the bilinear form $a(\cdot,\cdot)$ being defined by (\ref{eq3.1011});\\
\textbf{(ii)} \ We have
\begin{equation}\label{eq3.22}
    u_1^{k}(x,y)=-\sum_{j=1}^N\frac{\partial u^{k}_0}{\partial
   x_j}(x)\chi^j(y)\qquad\text{a.e. in } (x,y)\in\Omega\times Y^*,
\end{equation}
where $\chi^j$ is the solution to the cell
problem (\ref{eq3.9}).
\end{theorem}
\begin{proof}We present only the outlines since this proof is similar but less technical to that
of the case $M_{Y^*}(\rho)=0$.

Fix $k\geq 1$. By means of the minimax principle, as in
\cite{Vanni}, one easily proves the existence of a constant $C$
independent of $\varepsilon$ such that
$\lambda_\varepsilon^{k,+}<C$. Clearly, for fixed
$E\ni\varepsilon>0$, $u^{k,+}_\varepsilon$ lies in $V_\varepsilon$,
and
\begin{equation}\label{eq3.241}
\sum_{i,j=1}^N\int_{\Omega^\varepsilon}
a_{ij}(\frac{x}{\varepsilon})\frac{\partial
u^{k,+}_\varepsilon}{\partial x_j}
 \frac{\partial v}{\partial
 x_i}dx=\lambda^{k,+}_\varepsilon\int_{\Omega^\varepsilon}\rho(\frac{x}{\varepsilon}) u^{k,+}_\varepsilon
 v\,dx
\end{equation}
for any $v\in V_\varepsilon$. Bear in mind that
$\int_{\Omega^\varepsilon}\rho(\frac{x}{\varepsilon})(u^{k,+}_\varepsilon)^2
dx=1$ and choose $v=u^{k,+}_\varepsilon$ in (\ref{eq3.241}). The
boundedness of the sequence
$(\lambda^{k,+}_\varepsilon)_{\varepsilon\in E}$
 and the ellipticity assumption (\ref{eq1.2}) imply at once by means of Proposition \ref{p3.1} that the
 sequence $(P_\varepsilon u^{k,+}_\varepsilon)_{\varepsilon\in E}$ is bounded in $H^1_0(\Omega)$.
  Theorem \ref{t2.2} applies and
gives us $ \textbf{u}^{k}=(u_0^{k},u_1^{k})\in \mathbb{F}^1_0 $ such
that for some $\lambda_0^{k}\in\mathbb{R}$ and some subsequence
$E'\subset E$ we have (\ref{eq3.14})-(\ref{eq3.17}), where
(\ref{eq3.16}) is a direct consequence of (\ref{eq3.15}) by the
Rellich-Kondrachov theorem. For fixed $\varepsilon\in E'$, let
$\Phi_\varepsilon$ be as in Lemma \ref{l3.2}. Multiplying both sides
of  the first equality in (\ref{eq1.1}) by $\Phi_\varepsilon$ and
integrating over $\Omega^\varepsilon$ leads us to the variational
$\varepsilon$-problem
\begin{equation}\label{eq3.24}
\sum_{i,j=1}^N\int_{\Omega^\varepsilon}
a_{ij}(\frac{x}{\varepsilon})\frac{\partial P_\varepsilon
u^{k,+}_\varepsilon}{\partial x_j}
 \frac{\partial \Phi_\varepsilon}{\partial
 x_i}dx=\lambda^{k,+}_\varepsilon\int_{\Omega^\varepsilon}(P_\varepsilon u^{k,+}_\varepsilon)\rho(\frac{x}{\varepsilon})
 \Phi_\varepsilon\,dx.
\end{equation}
Sending $\varepsilon\in E'$ to $0$, keeping
(\ref{eq3.14})-(\ref{eq3.17}) and Lemma \ref{l3.2} in mind, we
obtain
\begin{equation*}
\sum_{i,j=1}^N\iint_{\Omega\times Y^*}a_{ij}(y)
   \left(\frac{\partial u_0^k}{\partial x_j}+
    \frac{\partial u_1^k}{\partial y_j}\right)\left(\frac{\partial \psi_0}{\partial x_i}+
    \frac{\partial \psi_1}{\partial
    y_i}\right)dxdy=\lambda^{k}_0\iint_{\Omega\times Y^*}u^{k}_0\psi_0(x)\rho(y)dxdy.
\end{equation*}
Therefore, $(\lambda^{k}_0,\textbf{u}^{k})\in\mathbb{R}\times
\mathbb{F}^1_0$ solves the following \textit{global homogenized
spectral problem}:

\begin{equation}\label{eq3.26}\left\{\begin{aligned}
  &\text{Find }(\lambda,\textbf{u})\in\mathbb{C}\times
\mathbb{F}^1_0 \text{ such that }\\&
   \sum_{i,j=1}^N\iint_{\Omega\times Y^*}a_{ij}(y)
   \left(\frac{\partial u_0}{\partial x_j}+
    \frac{\partial u_1}{\partial y_j}\right)\left(\frac{\partial \psi_0}{\partial x_i}+
    \frac{\partial \psi_1}{\partial
    y_i}\right)dxdy=\lambda M_{Y^*}(\rho)\int_\Omega u_0\psi_0
    \,dx\\&
   \text{for all } \Phi\in
\mathbb{F}^1_0,
   \end{aligned}\right.
\end{equation}
which leads to the macroscopic and microscopic problems
(\ref{eq3.18})-(\ref{eq3.21}) without any major difficulty.

As regards the normalization condition in (\ref{eq3.18}), we use the
decomposition $ \Omega^\varepsilon=Q^\varepsilon\cup
(\Omega^\varepsilon\setminus Q^\varepsilon)$ and the equality
$Q^\varepsilon=\Omega\cap\varepsilon G$. On the one hand, when $E'\ni
\varepsilon\to 0$,
$$
\int_{Q^\varepsilon}\rho(\frac{x}{\varepsilon})(P_\varepsilon u_\varepsilon^{k,+})(P_\varepsilon u_\varepsilon^{l,+})dx\to
M_{Y^*}(\rho)\int_\Omega u_0^{k} u_0^{l}dx,\quad
k,l=1,2,\cdots
$$
since
$$
\int_{Q^\varepsilon}\rho(\frac{x}{\varepsilon})(P_\varepsilon u_\varepsilon^{k,+})(P_\varepsilon u_\varepsilon^{l,+})\,dx=
\int_{\Omega}\chi_G(\frac{x}{\varepsilon})\rho(\frac{x}{\varepsilon})(P_\varepsilon u_\varepsilon^{k,+})(P_\varepsilon u_\varepsilon^{l,+})\,dx
$$
and
$(P_\varepsilon u_\varepsilon^{k,+})\chi_G^\varepsilon\rho^\varepsilon\rightharpoonup
M_{Y^*}(\rho)u_0^{k}$ in $L^2(\Omega)$-weak and
$P_\varepsilon u_\varepsilon^{l,+}\to u_0^{l}$ in $L^2(\Omega)$-strong as $E'\ni
\varepsilon\to 0$. On the other hand, the same line of reasoning as
in the proof of \cite[Proposition 3.6]{douanla2} leads to
\begin{equation}\label{eq3.261}
\lim_{E'\ni\varepsilon\to 0}\int_{\Omega^\varepsilon\setminus
Q^\varepsilon}\rho(\frac{x}{\varepsilon})(P_\varepsilon u_\varepsilon^{k,+})(P_\varepsilon u_\varepsilon^{l,+})\,dx=0
\end{equation}
The normalization condition in (\ref{eq3.18}) follows thereby. In
fact, we have just proved that $\{u_0^{k,+}\}_{k=1}^\infty$ is an
orthogonal basis in $L^2(\Omega)$.
\end{proof}

\begin{remark}\label{r3.2}
\begin{itemize}
\item The eigenfunctions $\{u_0^{k}\}_{k=1}^\infty$ are orthonormalized by
    $$
       \int_\Omega
       u_0^{k}u_0^{l}dx=\frac{\delta_{k,l}}{M_{Y^*}(\rho)}\quad
       k,l=1,2,3,\cdots
    $$
    \item If $\lambda_0^{k}$ is simple (this is the case for
$\lambda_0^{1}$), then
by Theorem \ref{t3.1}, $\lambda_\varepsilon^{k,+}$ is also simple, for small $\varepsilon$,
and we can choose the eigenfunctions $u_\varepsilon^{k,+}$ such that the convergence results (\ref{eq3.15})-(\ref{eq3.17}) hold for
the whole sequence $E$. In this case, the following corrector type result holds:
\[
\lim_{E\ni\varepsilon\to 0}\left\|D_x(P_\varepsilon u_\varepsilon^{k,+}(\cdot))-D_x u_0^k(\cdot)-D_y u_1^k(\cdot,\frac{.}{\varepsilon}) \right\|_{L^2(\Omega)^N}=0.
\]
\item Replacing $\rho$ with $-\rho$ in (\ref{eq1.1}), Theorem
 \ref{t3.1} also applies to the negative part of the spectrum in the
 case $M_{Y^*}(\rho)<0$.
\end{itemize}
\end{remark}

\subsubsection{Negative part of the spectrum}
We now investigate the negative part of the spectrum  $(\lambda_\varepsilon^{k,-}, u_\varepsilon^{k,-})_{\varepsilon\in E}$.
 Before we can do this we need a few preliminaries and stronger regularity hypotheses on $T$, $\rho$ and the coefficients $(a_{ij})_{i,j=1}^N$.
 We assume in this subsection that $\partial T$ is $C^{2,\delta}$ and $\rho$ and the coefficients
  $(a_{ij})_{i,j=1}^N$ are $\delta$-H\"{o}lder continuous ($0<\delta<1$).

The following spectral problem is well posed
\begin{equation} \label{eq3.39}
\left\{\begin{aligned} &\text{Find }
(\lambda,\theta)\in\mathbb{C}\times
H^1_{per}(Y^*)\\
&-\sum_{i,j=1}^N\frac{\partial}{\partial
y_j}\left(a_{ij}(y)\frac{\partial
\theta}{\partial y_i}\right)=\lambda\rho(y) \theta\ \text{ in }\ \ Y^*\\
&\sum_{i,j=1}^N a_{ij}(y)\frac{\partial \theta}{\partial y_j}n_i=0
\text{ on }\ \partial T
\end{aligned}\right.
\end{equation}
and possesses a spectrum with similar properties to that of (\ref{eq1.1}), two infinite
(positive and negative) sequences. We recall that (\ref{eq3.39}) admits a unique nontrivial
 eigenvalue having an eigenfunction with definite sign, the first negative one, since we have
  $M_{Y^*}(\rho)>0$ (see e.g., \cite{Brown, HessSenn}). In the sequel we will only make use
  of $(\lambda_1^-, \theta_1^-)$, the first negative eigencouple to (\ref{eq3.39}). After proper sign choice we assume
that
\begin{equation}\label{eq3.40}
    \theta_1^->0 \ \  \text{ in }\  \in Y^*.
\end{equation}
We also recall that $\theta_1^-$ is $\delta$-H\"{o}lder continuous(see e.g., \cite{Trudinger}),
hence can be extended to a $Y$-periodic function living in $L^\infty(\mathbb{R}^N_y)$ still denoted
by $\theta_1^-$. Notice that we have
\begin{equation}\label{eq3.391}
    \int_{Y^*}\rho(y)(\theta_1^{-}(y))^2dy<0,
\end{equation}
as is easily seen from the variational equality ( keep the
ellipticity hypothesis (\ref{eq1.2}) in mind)
$$
\sum_{i,j=1}^N\int_{Y^*}a_{ij}\frac{\partial \theta_1^-}{\partial y_j}\frac{\partial \theta_1^-}{\partial y_i}dy=
\lambda_1^-\int_{Y^*}\rho(y)(\theta_1^-(y))^2dy.
$$
Bear in mind that problem (\ref{eq3.39}) induces by a scaling
argument the following equalities:
\begin{equation} \label{eq3.40}
\left\{\begin{aligned} &-\sum_{i,j=1}^N\frac{\partial}{\partial
x_j}\left(a_{ij}(\frac{x}{\varepsilon})\frac{\partial
\theta^\varepsilon}{\partial x_i}\right)=\frac{1}{\varepsilon^2}\lambda\rho(\frac{x}{\varepsilon})
\theta(\frac{x}{\varepsilon})\text{ in } Q^\varepsilon\\
&\sum_{i,j=1}^N a_{ij}(\frac{x}{\varepsilon})\frac{\partial
\theta^\varepsilon}{\partial x_j}n_i(\frac{x}{\varepsilon})=0 \text{ on } \partial
Q^\varepsilon,
\end{aligned}\right.
\end{equation}
where $\theta^\varepsilon(x)=\theta(\frac{x}{\varepsilon})$.
However, $\theta^\varepsilon$ is not zero on $\partial \Omega$. We
now introduce the following spectral problem (with an indefinite
density function)
\begin{equation} \label{eq3.41}\left\{\begin{aligned}
\text{Find } (\xi_\varepsilon,v_\varepsilon)\in\mathbb{C}\times
V_\varepsilon&\\
-\sum_{i,j=1}^N\frac{\partial}{\partial
x_j}\left(\widetilde{a}_{ij}(\frac{x}{\varepsilon})\frac{\partial
v_\varepsilon}{\partial x_i}\right)&=\xi_\varepsilon\widetilde{\rho}(\frac{x}{\varepsilon}) v_\varepsilon(x)\text{ in } \Omega^\varepsilon\\
\sum_{i,j=1}^N
\widetilde{a}_{ij}(\frac{x}{\varepsilon})\frac{\partial
v_\varepsilon}{\partial x_j}n_i(\frac{x}{\varepsilon})&=0 \text{ on
} \partial T^\varepsilon\\v_\varepsilon(x)&=0  \text{ on }
\partial \Omega,
\end{aligned}\right.
\end{equation}
with new spectral eigencouple $(\xi_\varepsilon,
v_\varepsilon)\in\mathbb{C}\times V_\varepsilon$, where
$\widetilde{a}_{ij}(y)=(\theta_1^-)^2(y)a_{ij}(y)$ and
$\widetilde{\rho}(y)=(\theta_1^-)^2(y)\rho(y)$. Notice that $\widetilde{a}_{ij}\in L_{per}^\infty(Y)$ and
$\widetilde{\rho} \in L_{per}^\infty(Y)$. As
$0<c_-\leq \theta_1^-(y)\leq c^+<+\infty$ ($c_-, c^+\in\mathbb{R}$), the operator on the left hand side of
(\ref{eq3.41}) is uniformly elliptic and Theorem \ref{t3.1} applies to the negative part of the spectrum of
(\ref{eq3.41}) (see (\ref{eq3.391}) and Remark \ref{r3.2}). The effective spectral problem for (\ref{eq3.41}) reads
\begin{equation}\label{eq3.411}\left\{\begin{aligned}
 -\sum_{i,j=1}^N\frac{\partial}{\partial x_j}\left(\widetilde{q}_{ij}
 \frac{\partial v_0}{\partial x_i}\right)&=\xi_0 M_{Y^*}(\widetilde{\rho}) v_0\quad \text{in
 }\Omega\\v_0&=0\quad \text{ on  } \partial\Omega\\\int_{\Omega}|v_0|^2dx&=\frac{-1}{M_{Y^*}(\widetilde{\rho})}.
\end{aligned}\right.
\end{equation}
The effective coefficients $\{\widetilde{q}_{ij}\}_{1\leq i,j\leq N}$
 being defined as expected, i.e.,
\begin{equation}\label{eq3.42}
\widetilde{q}_{ij}=\int_{Y^*}\widetilde{a}_{ij}(y)dy-\sum_{l=1}^N\int_{Y^*}\widetilde{a}_{il}(y)\frac{\partial
\widetilde{\chi}^j}{\partial y_l}(y)dy,
\end{equation}
with $\widetilde{\chi}^l\in  H^1_{per}(Y^*)/\mathbb{R}$  $(l=1,...,N)$ being the
solution to the following local problem
\begin{equation} \label{eq3.43}\left\{\begin{aligned}
&\widetilde{\chi}^l\in  H^1_{per}(Y^*)/\mathbb{R}\\&
\sum_{i,j=1}^N\int_{Y^*}\widetilde{a}_{ij}(y)\frac{\partial
\widetilde{\chi}^l}{\partial y_j}\frac{\partial
v}{\partial
y_i}dy=\sum_{i=1}^N\int_{Y^*}\widetilde{a}_{il}(y)\frac{\partial
v}{\partial y_i}dy\\& \text{for all } v\in
 H^1_{per}(Y^*)/\mathbb{R}.\end{aligned}\right.
\end{equation}
We will use the following notation in the sequel:
\begin{equation} \label{eq3.4311}
\widetilde{a}(u,v)=\sum_{i,j=1}^N\int_{Y^*}\widetilde{a}_{ij}(y)\frac{\partial
u}{\partial y_j}\frac{\partial
v}{\partial
y_i}dy\qquad \left(u,v\in  H^1_{per}(Y^*)/\mathbb{R}\right).
\end{equation}
Notice that the spectrum of (\ref{eq3.411}) is as follows
$$
0>\xi_0^1>\xi_0^2\geq \xi_0^3\geq \cdots \geq\xi_0^j \geq \cdots   \to -\infty \text { as } j\to\infty.
$$
Making use of (\ref{eq3.40}), the same line of reasoning as in \cite[Lemma 6.1]{Vanni} shows that the negative
spectral parameters of problems (\ref{eq1.1}) and (\ref{eq3.41}) verify:
$$
u_\varepsilon^{k,-}=(\theta_1^{-})^\varepsilon v_\varepsilon^{k,-}\quad (\varepsilon\in E,\ k=1,2\cdots)
$$
and
$$
\lambda_\varepsilon^{k,-}=\frac{1}{\varepsilon^2}\lambda^-_1+\xi_\varepsilon^{k,-} + o(1), \quad (\varepsilon\in E,\ k=1,2\cdots).
$$
The presence of the term $o(1)$ is due to integrals over
$\Omega^\varepsilon\setminus Q^\varepsilon$,
 like the one in
(\ref{eq3.261}), which converge to zero with $\varepsilon$, remember that (\ref{eq3.40}) holds in $Q^\varepsilon$ but not
$\Omega^\varepsilon$. As will be seen below, the sequence $ (\xi_\varepsilon^{k,-})_{\varepsilon\in E} $ is
bounded in $\mathbb{R}$. In another words, $\lambda_\varepsilon^{k,- }$ is of order $1/\varepsilon^2$ and tends
to $-\infty$ as $\varepsilon$ goes to zero. It is now clear why the limiting behavior of negative eigencouples is
not straightforward as that of positive ones.

The suitable orthonormalization condition for
(\ref{eq3.41}) is the one the reader is expecting:
\begin{equation}\label{eq3.412}
    \int_{\Omega^\varepsilon}\widetilde{\rho}(\frac{x}{\varepsilon})v_\varepsilon^{k,-}v_\varepsilon^{l,-}dx=-\delta_{k,l}\quad
    k,l=1,2,\cdots
\end{equation}
We now state the homogenization theorem for the negative part of the spectrum of (\ref{eq1.1}).

\begin{theorem}\label{t3.2}
We assume that $\partial T$ is $C^{2,\delta}$ and $\rho$ and the coefficients
$(a_{ij})_{i,j=1}^N$ are $\delta$-H\"{o}lder continuous ($0<\delta<1$). For each $k\geq 1$ and each $\varepsilon\in E$, let
$(\lambda^{k,-}_\varepsilon,u^{k,-}_\varepsilon)$ be the $k^{th}$
negative eigencouple to (\ref{eq1.1}) with $M_{Y^*}(\rho)>0$ and
(\ref{eq3.412}). Then, there exists a subsequence $E'$ of $E$ such
that
\begin{eqnarray}
 \lambda^{k,-}_\varepsilon-\frac{\lambda_1^-}{\varepsilon^2} &\to&  \xi^k_0\quad\text{in }\ \mathbb{R}\ \text{ as }E\ni\varepsilon\to 0\label{eq3.44}\\
  P_\varepsilon v^{k,-}_\varepsilon&\to& v^{k}_0 \quad\text{in }\ \ H^1_0(\Omega)\text{-weak}\text{ as }E'\ni\varepsilon\to 0\label{eq3.45}\\
  P_\varepsilon  v^{k,-}_\varepsilon&\to& v^{k}_0 \quad\text{in }\ \ L^2(\Omega)\text{ as }E'\ni\varepsilon\to 0\label{eq3.46}\\
\frac{\partial P_\varepsilon  v^{k,-}_\varepsilon}{\partial
x_{j}}&\xrightarrow{2s} & \frac{\partial v_{0}^{k}}{\partial
x_{j}}+\frac{\partial v_{1}^{k}}{\partial y_{j}}\text{\ in
}L^{2}(\Omega)\text{ as }E'\ni\varepsilon\to 0\ (1\leq j\leq
N)\label{eq3.47}
\end{eqnarray}
where $(\xi^{k}_0,v^{k}_0)\in \mathbb{R}\times H^1_0(\Omega)$ is the
 $k^{th}$ eigencouple to the
 spectral problem
\begin{equation}\label{eq3.48}\left\{\begin{aligned}
 -\sum_{i,j=1}^N\frac{\partial}{\partial x_i}\left(\frac{1}{M_{Y^*}(\widetilde{\rho})}\widetilde{q}_{ij}
 \frac{\partial v_0}{\partial x_j}\right)&=\xi_0 v_0\quad \text{in
 }\Omega\\v_0&=0\quad \text{ on  } \partial\Omega\\\int_{\Omega}|v_0|^2dx&=\frac{-1}{M_{Y^*}(\widetilde{\rho})},
\end{aligned}\right.
\end{equation}
$v_1^{k}\in L^2(\Omega;H^{1,*}_{per}(Y))$ and where the coefficients $\{\widetilde{q}_{ij}\}_{1\leq i,j\leq N}$ are defined
by (\ref{eq3.42}). Moreover, for almost every $x\in\Omega$ the following hold true:\\
\textbf{(i)} \ The restriction to $Y^*$ of $v_1^{k}(x)$ is the solution to the
variational problem
\begin{equation}\label{eq3.49}\left\{\begin{aligned}
  &v_1^{k}(x)\in  H^1_{per}(Y^*)/\mathbb{R}\\&
  \widetilde{a}(v_1^{k}(x),u)=-\sum_{i,j=1}^N \frac{\partial v^{k}_0}{\partial
  x_j}  \int_{Y^*}\widetilde{a}_{ij}(y)\frac{\partial u}{\partial y_i}dy \\&
   \forall u\in  H^1_{per}(Y^*)/\mathbb{R},
   \end{aligned}\right.
\end{equation}
the bilinear form $\widetilde{a}(\cdot,\cdot)$ being defined by (\ref{eq3.4311});\\
\textbf{(ii)} \ We have
\begin{equation}\label{eq3.50}
    v_1^{k}(x,y)=-\sum_{j=1}^N\frac{\partial v^{k}_0}{\partial
   x_j}(x)\widetilde{\chi}^j(y)\qquad\text{a.e. in } (x,y)\in\Omega\times Y^*,
\end{equation}
where $\widetilde{\chi}^j$ is the solution to the cell problem (\ref{eq3.43}).
\end{theorem}
\begin{remark}
\begin{itemize}
\item The eigenfunctions $\{v_0^{k}\}_{k=1}^\infty$ are orthonormalized by
    $$
       \int_\Omega
       v_0^{k}v_0^{l}dx=\frac{-\delta_{k,l}}{M_{Y^*}(\widetilde{\rho})}\quad
       k,l=1,2,3,\cdots
    $$
    \item If $\xi_0^{k}$ is simple (this is the case for
$\xi_0^{1}$), then
by Theorem \ref{t3.2}, $\lambda_\varepsilon^{k,-}$ is also simple, for small $\varepsilon$,
and we can choose the `eigenfunctions' $v_\varepsilon^{k,-}$ such that the convergence results (\ref{eq3.45})-(\ref{eq3.47}) hold for
the whole sequence $E$. In this case, the following corrector type result holds:
\[
\lim_{E\ni\varepsilon\to 0}\left\|D_x(P_\varepsilon v_\varepsilon^{k,-}(\cdot))-D_x v_0^{k}(\cdot)-D_y v_1^{k}(\cdot,\frac{.}{\varepsilon})
\right\|_{L^2(\Omega)^N}=0.
\]
\item Replacing $\rho$ with $-\rho$ in (\ref{eq1.1}), Theorem
 \ref{t3.2} adapts to the positive part of the spectrum in the
 case $M_{Y^*}(\rho)<0$.
\end{itemize}
\end{remark}

\subsection{The case $M_{Y^*}(\rho)=0$}
We prove a homogenization result for both the positive part and the
negative part of the spectrum simultaneously. As will be clear in the proof of Theorem \ref{t3.3} below, we assume in this case
that the eigenfunctions are orthonormalized as follows
\begin{equation}\label{eq3.22}
   \int_{\Omega^\varepsilon}\rho(\frac{x}{\varepsilon})u_\varepsilon^{k,\pm}u_\varepsilon^{l,\pm}dx=\pm\varepsilon\delta_{k,l}\quad
    k,l=1,2,\cdots
\end{equation}
Let $\chi^0$ be the solution to $(\ref{eq3.91})$ and put
\begin{equation}\label{eq3.23}
    \nu^2=\sum_{i,j=1}^N\int_{Y^*}a_{ij}(y)\frac{\partial \chi^0}{\partial y_j}\frac{\partial \chi^0}{\partial
    y_i}dy.
\end{equation}
Indeed, the right hand side of (\ref{eq3.23}) is positive. We now
recall that the following spectral problem for a quadratic operator
pencil with respect to $\nu$,
\begin{equation}\label{eq3.24}
\left\{\begin{aligned} -\sum_{i,j=1}^N\frac{\partial}{\partial
x_j}\left(q_{ij}\frac{\partial
u_0}{\partial x_i}\right)&=\lambda_0^2\nu^2 u_0\text{ in } \Omega\\
u_0&=0 \text{ on } \partial \Omega,
\end{aligned}\right.
\end{equation}
has a spectrum consisting of two infinite sequences
$$
0<\lambda_0^{1,+}< \lambda_0^{2,+}\leq \cdots \leq
\lambda_0^{k,+}\leq \dots,\quad \lim_{k\to
+\infty}\lambda_0^{k,+}=+\infty
$$
and
$$
0>\lambda_0^{1,-}> \lambda_0^{2,-}\geq \cdots \geq
\lambda_0^{k,-}\geq \dots,\quad \lim_{k\to
+\infty}\lambda_0^{k,-}=-\infty.
$$
with $\lambda^{k,+}_0=-\lambda^{k,-}_0\ \ k=1,2,\cdots$ and with the
corresponding eigenfunctions $u_0^{k,+}=u_0^{k,-}$. We note by
passing that $\lambda^{1,+}_0$ and $\lambda^{1,-}_0$ are simple. We
are now in a position to state the homogenization result in the
present case.

\begin{theorem}\label{t3.3}
We assume that $\Omega$ and $T$ have $\mathcal{C}^1$ boundaries. For each $k\geq 1$ and each $\varepsilon\in E$, let
$(\lambda^{k,\pm}_\varepsilon,u^{k,\pm}_\varepsilon)$ be the
$(k,\pm)^{th}$ eigencouple to (\ref{eq1.1}) with $M_{Y^*}(\rho)=0$
and (\ref{eq3.22}). Then, there exists a subsequence $E'$ of $E$
such that
\begin{eqnarray}
 \varepsilon\lambda^{k,\pm}_\varepsilon &\to&  \lambda^{k,\pm}_0\quad\text{in }\ \mathbb{R}\ \text{ as }E\ni\varepsilon\to 0\label{eq3.25}\\
  P_\varepsilon u^{k,\pm}_\varepsilon&\to& u^{k,\pm}_0 \quad\text{in }\ \ H^1_0(\Omega)\text{-weak}\text{ as }E'\ni\varepsilon\to 0\label{eq3.26}\\
  P_\varepsilon u^{k,\pm}_\varepsilon&\to& u^{k,\pm}_0 \quad\text{in }\ \ L^2(\Omega)\text{ as }E'\ni\varepsilon\to 0\label{eq3.27}\\
\frac{\partial P_\varepsilon u^{k,\pm}_\varepsilon}{\partial
x_{j}}&\xrightarrow{2s} & \frac{\partial u_{0}^{k,\pm}}{\partial
x_{j}}+\frac{\partial u_{1}^{k,\pm}}{\partial y_{j}}\text{\ in
}L^{2}(\Omega)\text{ as }E'\ni\varepsilon\to 0\ (1\leq j\leq
N)\label{eq3.28}
\end{eqnarray}
where $(\lambda^{k,\pm}_0,u^{k,\pm}_0)\in \mathbb{R}\times
H^1_0(\Omega)$ is the $(k,\pm)^{th}$ eigencouple to the
following spectral problem
for a quadratic operator pencil with respect to $\nu$,
\begin{equation}\label{eq3.29}
\left\{\begin{aligned} -\sum_{i,j=1}^N\frac{\partial}{\partial
x_i}\left(q_{ij}\frac{\partial
u_0}{\partial x_j}\right)&=\lambda_0^2\nu^2 u_0\text{ in } \Omega\\
u_0&=0 \text{ on } \partial \Omega,
\end{aligned}\right.
\end{equation}
$u_1^{k,\pm}\in L^2(\Omega;H^{1,*}_{per}(Y))$ and where the coefficients $\{q_{ij}\}_{1\leq i,j\leq N}$ are defined by
(\ref{eq3.10}). We have the following normalization condition
\begin{equation}\label{eq2.291}
    \int_{\Omega}|u_0^{k,\pm}|^2\,dx=\frac{\pm 1}{\lambda_0^{k,\pm} \nu^2}\qquad k=1,2,\cdots
\end{equation}
Moreover, for almost every $x\in\Omega$ the following hold true:\\
\textbf{(i)} \ The restriction to $Y^*$ of $u_1^{k,\pm}(x)$ is the solution to the
variational problem
\begin{equation}\label{eq3.30}\left\{\begin{aligned}
  &u_1^{k,\pm}(x)\in   H^1_{per}(Y^*)/\mathbb{R}\\&
  a(u_1^{k,\pm}(x),v)=\lambda^{k,\pm}_0 u_0^{k,\pm}(x)\int_{Y^*}\rho(y)v(y)dy-\sum_{i,j=1}^N \frac{\partial u^{k,\pm}_0}{\partial
  x_j}(x)  \int_{Y^*}a_{ij}(y)\frac{\partial v}{\partial y_i}dy \\&
   \forall v\in  H^1_{per}(Y^*)/\mathbb{R},
   \end{aligned}\right.
\end{equation}
the bilinear form $ a(\cdot,\cdot)$ being defined by (\ref{eq3.1011});\\
\textbf{(ii)} \ We have
\begin{equation}\label{eq3.31}
   u_1^{k,\pm}(x,y)=\lambda^{k,\pm}_0 u_0^{k,\pm}(x)\chi^0(y)-\sum_{j=1}^N\frac{\partial u^{k,\pm}_0}{\partial
   x_j}(x)\chi^j(y)\quad\text{a.e. in } (x,y)\in\Omega\times Y^*,
\end{equation}
where $\chi^j\ (1\leq j\leq N)$ and $\chi^0$ are the solutions to the cell problems (\ref{eq3.9}) and
(\ref{eq3.91}), respectively.
\end{theorem}

\begin{proof}
Fix $k\geq 1$, using the minimax principle, as in \cite{Vanni}, we
get a constant $C$ independent of $\varepsilon$ such that
$|\varepsilon\lambda_\varepsilon^{k,\pm}|<C$. We have
$u^{k,\pm}_\varepsilon\in V_\varepsilon$ and
\begin{equation}\label{eq3.32}
\sum_{i,j=1}^N\int_{\Omega^\varepsilon}
a_{ij}(\frac{x}{\varepsilon})\frac{\partial
  u^{k,\pm}_\varepsilon}{\partial x_j}
 \frac{\partial v}{\partial
 x_i}dx=(\varepsilon\lambda^{k,\pm}_\varepsilon)\frac{1}{\varepsilon}\int_{\Omega^\varepsilon}  \rho(\frac{x}{\varepsilon})u^{k,\pm}_\varepsilon
 v\,dx
\end{equation}
for any $v\in V_\varepsilon$. Bear in mind that
$\int_{\Omega^\varepsilon}\rho(\frac{x}{\varepsilon})(u^{k,\pm}_\varepsilon)^2
dx=\pm\varepsilon$ and choose $v=u^{k,\pm}_\varepsilon$ in (\ref{eq3.32}).
The boundedness of the sequence $(\varepsilon\lambda^{k,\pm}_\varepsilon)_{\varepsilon\in E}$
 and the ellipticity assumption (\ref{eq1.2}) imply at once by means of Proposition \ref{p3.1} that the
 sequence $(P_\varepsilon u^{k,\pm}_\varepsilon)_{\varepsilon\in E}$ is bounded in $H^1_0(\Omega)$.
  Theorem \ref{t2.2} applies and
gives us $ \textbf{u}^{k,\pm}=(u_0^{k,\pm},u_1^{k,\pm})\in
\mathbb{F}^1_0 $ such that for some $\lambda_0^{k,\pm}\in\mathbb{R}$
and some subsequence $E'\subset E$ we have
(\ref{eq3.25})-(\ref{eq3.28}), where (\ref{eq3.27}) is a direct
consequence of (\ref{eq3.26}) by the Rellich-Kondrachov theorem. For
fixed $\varepsilon\in E'$, let $\Phi_\varepsilon$ be as in Lemma
\ref{l3.2}. Multiplying both sides of  the first equality in
(\ref{eq1.1}) by $\Phi_\varepsilon$ and integrating over
$\Omega^\varepsilon$ leads us to the variational
$\varepsilon$-problem
\begin{equation*}
\sum_{i,j=1}^N\int_{\Omega^\varepsilon}
a_{ij}(\frac{x}{\varepsilon})\frac{\partial P_\varepsilon
u^{k,\pm}_\varepsilon}{\partial x_j}
 \frac{\partial \Phi_\varepsilon}{\partial
 x_i}dx=(\varepsilon\lambda^{k,\pm}_\varepsilon)\frac{1}{\varepsilon}\int_{\Omega^\varepsilon}(P_\varepsilon u^{k,\pm}_\varepsilon)
 \rho(\frac{x}{\varepsilon})
\Phi_\varepsilon\,dx.
\end{equation*}
Sending $\varepsilon\in E'$ to $0$, keeping
(\ref{eq3.25})-(\ref{eq3.28}) and Lemma \ref{l3.2} in mind, we
obtain
\begin{equation}\label{eq3.33}
a_\Omega(\mathbf{u}^{k,\pm},\Phi)=\lambda^{k,\pm}_0\iint_{\Omega\times
Y^*}\left(u_1^{k,\pm}(x,y)\psi_0(x)\rho(y)+
    u_0^{k,\pm}\psi_1(x,y)\rho(y)\right)dxdy
    \end{equation}
The right-hand side follows as explained below. Using the
decomposition
$\Omega^\varepsilon=Q^\varepsilon\cup(\Omega^\varepsilon\setminus
Q^\varepsilon)$ and the equality
$Q^\varepsilon=\Omega\cap\varepsilon G$ we arrive at

\begin{eqnarray*}
   \frac{1}{\varepsilon}\int_{\Omega^\varepsilon}(P_\varepsilon
u^{k,\pm}_\varepsilon)\rho(\frac{x}{\varepsilon})
\Phi_\varepsilon\,dx&=&\frac{1}{\varepsilon}\int_{\Omega}(P_\varepsilon
u^{k,\pm}_\varepsilon)\psi_0(x)
 \rho(\frac{x}{\varepsilon})\chi_G(\frac{x}{\varepsilon})
 \,dx \\
   &+& \int_{\Omega}(P_\varepsilon u^{k,\pm}_\varepsilon)\psi_1(x,\frac{x}{\varepsilon})
 \rho(\frac{x}{\varepsilon})\chi_G(\frac{x}{\varepsilon})
 \,dx+o(1).
\end{eqnarray*}
On the one hand we have
$$
\lim_{E'\ni\varepsilon\to 0}\int_{\Omega}(P_\varepsilon
u^{k,\pm}_\varepsilon)\psi_1(x,\frac{x}{\varepsilon})
 \rho(\frac{x}{\varepsilon})\chi_G(\frac{x}{\varepsilon})
 \,dx=\iint_{\Omega\times
Y} u_0^{k,\pm}\psi_1(x,y)\rho(y)\chi_G(y)\,dxdy.
$$
On the other hand, owing to (\ref{eq1.14}) of Theorem \ref{t2.2}, the following holds:
$$
\lim_{E'\ni\varepsilon\to 0}\frac{1}{\varepsilon}\int_{\Omega}
(P_\varepsilon u^{k,\pm}_\varepsilon)\psi_0(x)
 \rho(\frac{x}{\varepsilon})\chi_G(\frac{x}{\varepsilon})
 \,dx=\iint_{\Omega\times
Y}u_1^{k,\pm}(x,y)\psi_0(x)\rho(y)\chi_G(y)\,dxdy.
$$
Indeed $\rho\chi_G\in L^2_{per}(Y)/\mathbb{R}$ as we clearly have $ \int_Y
\rho(y)\chi_G(y)dy=\int_{Y^*}\rho(y)dy=0 $. We have just proved that
$(\lambda^{k,\pm}_0,\textbf{u}^{k,\pm})\in\mathbb{R}\times
\mathbb{F}^1_0$ solves the following \textit{global homogenized
spectral problem}:
\begin{equation}\label{eq3.34}\left\{\begin{aligned}
  &\text{Find }(\lambda,\textbf{u})\in\mathbb{C}\times
\mathbb{F}^1_0 \text{ such that }\\&
  a_\Omega(\mathbf{u},\Phi)=\lambda\iint_{\Omega\times
Y^*}\left(u_1(x,y)\psi_0(x)\rho(y)+
    u_0\psi_1(x,y)\rho(y)\right)dxdy \\&
   \text{for all } \Phi\in
\mathbb{F}^1_0 .
   \end{aligned}\right.
\end{equation}
To prove (i), choose $\Phi=(\psi_0,\psi_1)$ in (\ref{eq3.33})  such
that $\psi_{0}=0$ and $\psi_{1}=\varphi\otimes v_1$, where
$\varphi\in\mathcal{D}(\Omega)$ and $v_1\in  H^1_{per}(Y^*)/\mathbb{R}$ to get
\[
\int_{\Omega}\varphi(x)\left[\sum_{i,j=1}^N\int_{Y^*}a_{ij}(y)\left(\frac{\partial
u^{k,\pm}_0}{\partial x_j}+\frac{\partial u^{k,\pm}_1}{\partial y_j}
\right)\frac{\partial v_1}{\partial
y_i}dy\right]dx=\int_\Omega\varphi(x)\left[\lambda_0^{k,\pm}u_0^{k,\pm}(x)\int_{Y^*}v_1(y)\rho(y)dy\right]dx
\]
Hence by the arbitrariness of $\varphi$, we have a.e. in $\Omega$
\[
\sum_{i,j=1}^N\int_{Y^*}a_{ij}(y)\left(\frac{\partial
u^{k,\pm}_0}{\partial x_j}+\frac{\partial u^{k,\pm}_1}{\partial y_j}
\right)\frac{\partial v_1}{\partial
y_i}dy=\lambda_0^{k,\pm}u_0^{k,\pm}(x)\int_{Y^*}v_1(y)\rho(y)dy
\]
for any  $v_1$ in  $ H^1_{per}(Y^*)/\mathbb{R}$, which is nothing but (\ref{eq3.30}).

Fix $x\in\overline{\Omega}$, multiply both sides of (\ref{eq3.9}) by
$-\frac{\partial u_0^{k,\pm}}{\partial x_j}(x)$ and sum over $1\leq
j\leq N$. Adding side by side to the resulting equality that
obtained after multiplying both sides of (\ref{eq3.91}) by
$\lambda_0^{k,\pm}u_0^{k,\pm}(x)$, we realize that
$z(x)=-\sum_{j=1}^N\frac{\partial u_0^{k,\pm}}{\partial
x_j}(x)\chi^j(y)+\lambda_0^{k,\pm}u_0^{k,\pm}(x)\chi^0(y)$ solves
(\ref{eq3.30}). Hence
\begin{equation}\label{eq3.36}
     u_1^{k,\pm}(x,y)=\lambda^{k,\pm}_0 u_0^{k,\pm}(x)\chi^0(y)-\sum_{j=1}^N\frac{\partial u^{k,\pm}_0}{\partial
   x_j}(x)\chi^j(y)\quad\text{a.e. in } (x,y)\in\Omega\times Y^*.
\end{equation}
by uniqueness of the solution to the variational problem
(\ref{eq3.30}). Thus (\ref{eq3.31}).

Considering now $\Phi=(\psi_0,\psi_1)$ in (\ref{eq3.33})  such that
$\psi_{0}\in\mathcal{D}(\Omega)$ and $\psi_{1}=0$ we get
$$
\sum_{i,j=1}^N\iint_{\Omega\times Y^*}a_{ij}(y)\left(\frac{\partial
u_0^{k,\pm}}{\partial x_j}+\frac{\partial u_1^{k,\pm}}{\partial
y_j}\right)\frac{\partial \psi_0}{\partial
x_i}dxdy=\lambda_0^{k,\pm}\iint_{\Omega\times
Y^*}u_1^{k,\pm}(x,y)\rho(y)\psi_0(x)dxdy,
$$
which by means of (\ref{eq3.36}) leads to
\begin{eqnarray}
    &&\sum_{i,j=1}^N  \int_{\Omega}q_{ij}\frac{\partial
u_0^{k,\pm}}{\partial x_j}\frac{\partial \psi_0}{\partial
x_i}dx+\lambda_0^{k,\pm}\sum_{i,j=1}^N\int_\Omega
u_0^{k,\pm}(x)\frac{\partial \psi_0}{\partial
x_i}dx\left(\int_{Y^*}a_{ij}(y)\frac{\partial \chi^0}{\partial y_j
}(y)dy\right)\nonumber\\
&&=-\lambda_0^{k,\pm}\sum_{j=1}^N\int_\Omega \frac{\partial
u_0^{k,\pm}}{\partial x_j}
\psi_0(x)dx\left(\int_{Y^*}\rho(y) \chi^j(y)
dy\right)\label{eq3.37}\\
&&+(\lambda_0^{k,\pm})^2\int_\Omega
u_0^{k,\pm}(x)\psi_0(x)dx\left(\int_{Y^*}\rho(y)\chi^0(y)dy\right)\nonumber.
\end{eqnarray}
Choosing $\chi^l\ (1\leq l\leq N)$ as test function in
(\ref{eq3.91}) and $\chi^0$ as test function in (\ref{eq3.9}) we
observe that
$$
\sum_{j=1}^N\int_{Y^*}a_{lj}(y)\frac{\partial \chi^0}{\partial y_j
}(y)dy=\int_{Y^*}\rho(y) \chi^l(y)dy=a(\chi^l,\chi^0)\quad
(l=1,\cdots N).
$$
Thus, in (\ref{eq3.37}), the second term in the left-hand side is
equal to the first one in the right-hand side. This leaves us with
\begin{equation}\label{eq3.38}
\int_{\Omega}q_{ij}\frac{\partial u_0^{k,\pm}}{\partial
x_j}\frac{\partial \psi_0}{\partial
x_i}dx=(\lambda_0^{k,\pm})^2\int_\Omega
u_0^{k,\pm}(x)\psi_0(x)dx\left(\int_{Y^*}\rho(y)\chi^0(y)dy\right).
\end{equation}
Choosing $\chi^0$ as test function in (\ref{eq3.91}) reveals that
$$
\int_{Y^*}\rho(y)\chi^0(y)dy=a(\chi^0,\chi^0)=\nu^2.
$$
Hence
$$
\sum_{i,j=1}^N\int_{\Omega}q_{ij}\frac{\partial
u_0^{k,\pm}}{\partial x_j}\frac{\partial \psi_0}{\partial
x_i}dx=(\lambda_0^{k,\pm})^2\nu^2\int_\Omega
u_0^{k,\pm}(x)\psi_0(x)dx,
$$
and
$$
-\sum_{i,j=1}^N\frac{\partial}{\partial
x_i}\left(q_{ij}\frac{\partial u_0^{k,\pm}}{\partial
x_j}(x)\right)=(\lambda_0^{k,\pm})^2\nu^2 u_0^{k,\pm}(x)\text{ in }
\Omega.
$$
Thus the convergence (\ref{eq3.25}) holds for the whole sequence
$E$.
As regards (\ref{eq2.291}), we notice that for fixed $k\geq 1$ and any $\phi\in\mathcal{D}(\Omega)$ one has (keep (\ref{eq1.14}) in mind)
$$
\lim_{E'\ni\varepsilon\to 0}\frac{1}{\varepsilon}\int_{\Omega}(P_\varepsilon u_{\varepsilon}^{k,\pm})\phi(x)
\rho(\frac{x}{\varepsilon})\chi_G(\frac{x}{\varepsilon})dx=\iint_{\Omega\times Y^*}u_1^{k,\pm}(x,y)\phi(x)\rho(y)\,dxdy.
$$
Hence, as $E'\ni\varepsilon\to 0$
 $$\frac{1}{\varepsilon}(P_\varepsilon  u_{\varepsilon}^{k,\pm})\rho^\varepsilon\chi_G^\varepsilon\rightharpoonup
 \int_{Y^*}u_1^{k,\pm}(\cdot,y)\rho(y)\, dy\quad \text{in }L^2(\Omega)-\text{weak}.
 $$
Using once again the decomposition $ \Omega^\varepsilon=Q^\varepsilon\cup
(\Omega^\varepsilon\setminus Q^\varepsilon)$ and the equality
$Q^\varepsilon=\Omega\cap\varepsilon G$, we get as $E'\ni\varepsilon\to 0$
$$
\frac{1}{\varepsilon}\int_{\Omega^\varepsilon}(P_\varepsilon u_{\varepsilon}^{k,\pm})(P_\varepsilon u_{\varepsilon}^{l,\pm})
\rho(\frac{x}{\varepsilon})\,dx\to
\iint_{\Omega\times Y^*}u_1^{k,\pm}(x,y)u_0^{l,\pm}(x)\rho(y)\,dxdy,
$$
for fixed $l\geq 1$. This together with (\ref{eq3.22}) and (\ref{eq3.36}) yields
\begin{equation}\label{eq4}
    \lambda_0^{k,\pm}\nu^2\int_{\Omega}u_0^{l,\pm}u_0^{k,\pm}\,dx-\sum_{j=1}^{N}a(\chi^j,\chi^0)\int_\Omega\frac{\partial
    u_0^{k,\pm}}{\partial x_j}u_0^{l,\pm}\,dx=\pm\delta_{k,l},\quad k,l=1,2,\cdots
\end{equation}
If $k=l$, then by Green's formula the sum in the left-hand side
vanishes and (\ref{eq4}) reduces to the desired result. This
concludes the proof.
\end{proof}
\begin{remark}
\begin{itemize}
      \item Permuting $k$ and $l$ in (\ref{eq4}) and adding side by side the resulting equality to (\ref{eq4}) we
      realize that the eigenfunctions $\{u_0^{k,\pm}\}_{k=1}^\infty$ are orthonormalized by
      $$
    \int_{\Omega}u_0^{l,\pm}(x)u_0^{k,\pm}(x)dx=\frac{\pm 2\delta_{k,l}}{\nu^2(\lambda_0^{k,\pm}+\lambda_0^{l,\pm})}\qquad  k,l=1,2,\cdots
      $$
     \item If $\lambda_0^{k,\pm}$ is simple (this is the case for
$\lambda_0^{1,\pm}$), then by Theorem \ref{t3.3},
$\lambda_\varepsilon^{k,\pm}$ is also simple, for small
$\varepsilon$, and we can choose the eigenfunctions
$u_\varepsilon^{k,\pm}$ such that the convergence results
(\ref{eq3.26})-(\ref{eq3.28}) hold for the whole sequence $E$. In
this case, the following corrector type result holds:
\[
\lim_{E\ni\varepsilon\to 0}\left\|D_x(P_\varepsilon u_\varepsilon^{k,\pm}(\cdot))-D_x u_0^{k,\pm}(\cdot)-D_y u_1^{k,\pm}
(\cdot,\frac{.}{\varepsilon}) \right\|_{L^2(\Omega)^N}=0.
\]
\end{itemize}
\end{remark}
\subsection*{Acknowledgments} The author is grateful to Dr. Jean Louis Woukeng for helpful
discussions.

\end{document}